\documentclass[11pt, a4paper]{amsart}
\usepackage{amssymb, array, amsmath, amscd, pdfpages, enumerate, amsthm, fixltx2e, setspace, hyperref, bbm}

\DeclareMathOperator*{\Span}{span}

\newcommand{\RR}{\mathbb{R}}
\newcommand{\CC}{\mathbb{C}}

\newcommand{\FF}{\mathbb{F}}

\newtheorem{thm}{Theorem}[section]
\newtheorem{prp}[thm]{Proposition}

\newtheorem{cor}[thm]{Corollary}
\theoremstyle{definition}
\newtheorem{rem}[thm]{Remark}
\newtheorem{ex}[thm]{Example}

\numberwithin{equation}{section}

\begin{document}

\title[Orderings in the method of alternating projections]{Non-optimality of the Greedy Algorithm for subspace orderings in the method of alternating projections}

\author{O. Darwin, A. Jha, S. Roy, D. Seifert, R. Steele, L. Stigant}
\address{Mathematical Institute\\
University of Oxford\\
Andrew Wiles Building\\
Radcliffe Observatory Quarter\\
Woodstock Road\\
Oxford OX2 6GG}
\email{oscar.darwin@magd.ox.ac.uk}
\email{aashraya.jha@univ.ox.ac.uk}
\email{souktik.roy@sjc.ox.ac.uk}
\email{david.seifert@sjc.ox.ac.uk}
\email{rhys.steele@sjc.ox.ac.uk}
\email{liam.stigant@jesus.ox.ac.uk}

\begin{abstract}
The method of alternating projections involves projecting an element of a Hilbert space cyclically onto a collection of closed subspaces. It is known that the resulting sequence always converges in norm and that one can obtain estimates for the rate of convergence in terms of quantities  describing the geometric relationship between the subspaces in question, namely their pairwise Friedrichs numbers. We consider the question of how best to order a given collection of subspaces so as to obtain the best estimate on the rate of convergence. We prove, by relating the ordering problem to a variant of the famous Travelling Salesman Problem, that correctness of a natural form of the Greedy Algorithm would imply that $\mathrm{P}=\mathrm{NP}$, before presenting a simple example which shows that, contrary to a claim made in the influential paper \cite{KaWe88}, the result of the Greedy Algorithm is not in general optimal. We go on to  establish sharp estimates on the degree to which the result of the Greedy Algorithm can differ from the optimal result. Underlying all of these results is a construction which shows that for any matrix whose entries satisfy certain natural assumptions it is possible to construct a Hilbert space and a collection of closed subspaces such that the pairwise Friedrichs numbers between the subspaces are given precisely by the entries of that matrix.

\end{abstract}


\maketitle

\section{Introduction}\label{sec:intro} 

Let $X$ be a real or complex Hilbert space, $N\ge2$ an integer, and suppose that $M_1,\dotsc,M_N$ are closed subspaces of $X$. Furthermore let $P_k$ denote the orthogonal projection onto $M_k$, $1\le k\le N$, and let $P_M$ denote the orthogonal projection onto the intersection  $M=M_1\cap\dotsc\cap M_N$. If we let $T=P_N\cdots P_1$ then it follows from a classical theorem due to Halperin \cite{Hal62} that 
\begin{equation}\label{eq:conv}
\|T^nx-P_Mx\|\to0,\quad n\to\infty,
\end{equation}
for all $x\in X$. It follows easily that, for any $x\in X$, the sequence in $X$ obtained by starting at $X$ and then projecting cyclically onto the $N$ subspaces $M_1,\dotsc,M_N$ must converge to the point $P_Mx$, which is the point in $M$ closest to the starting vector $x$. This procedure is known as the \emph{method of alternating projections} and has many applications, for instance to the iterative solution of large linear systems but also in the theory of partial differential equations and in image restoration; see \cite{De92} for a survey. 

In view of these applications it is important to understand the \emph{rate} at which the convergence in \eqref{eq:conv} takes place; see for instance \cite{BGM11, BaSe16, DeHu10, DeHu15} for in-depth investigations.  Recall that the \emph{Friedrichs number} $c(L_1,L_2)$ between the two subspaces $L_1, L_2$ of $X$ is defined as
$$c(L_1, L_2)=\sup\big\{|(x_1,x_2)|:x_k\in L_k\cap L^\perp\mbox{ and }\|x_k\|\le1\mbox{ for }k=1,2\big\},$$
where $L=L_1\cap L_2$. The Friedrichs number lies in the interval $[0,1]$ and may be thought of as the cosine of the `angle' between the subspaces $L_1$ and $L_2$. It is shown in \cite[Theorem~2]{KaWe88} that for $N=2$ in the method of alternating projections we have   
\begin{equation}\label{eq:KW}
\|T^n-P_M\|=c(M_1,M_2)^{2n-1},\quad n\ge1.
\end{equation}
 When $N\ge3$ no sharp upper bound of this form is known, but it is shown in \cite[Corollary~2.10]{DeHu97} that
\begin{equation}\label{eq:prod}
\|T^n-P_M\|\le c(M_N,M_{N-1})^n\cdots c(M_2,M_1)^n c(M_1,M_N)^{n-1},\quad n\ge1,
\end{equation} 
provided the subspaces are \emph{pairwise quasi-disjoint} in the sense that $M_k\cap M_\ell\cap M^\perp=\{0\}$ for $1\le k,\ell\le N$ with $k\ne\ell$. Moreover, the assumption on the subspaces cannot be omitted. The same bound was obtained earlier in \cite{KaWe88} in the special case where the subspaces $M_1\cap M^\perp,\dots,M_N\cap M^\perp$ are \emph{independent}, which is to say that if vectors  $x_k\in M_k\cap M^\perp$, $1\le k\le N$, satisfy $x_1+\dots +x_N=0$ then $x_1=\dotsc= x_N=0$. 

Examples in \cite[Section~3]{DeHu97} show both that the bound in \eqref{eq:bound} fails to be sharp in some special cases, thus disproving a conjecture made in \cite{KaWe88}, and more generally that it is not possible for $N\ge3$ to obtain a sharp upper bound for $\|T^n-P_M\|$, $n\ge1$, which depends only on the pairwise Friedrichs numbers between the subspaces $M_1,\dots,M_N$. Nevertheless, the estimate in \eqref{eq:prod} recovers the sharp bound in \eqref{eq:KW} when $N=2$ and  holds with equality in a number of other cases, for instance if all of the spaces $M_1,\dotsc,M_N$ are one-dimensional.  We also see from \eqref{eq:prod} that if the Friedrichs number between a pair of consecutive subspaces is zero then we have convergence in the method of alternating projections after at most two steps. Since our interest here is primarily in the asymptotic rate of convergence as $n\to\infty$, there is no significant loss of generality in assuming that  $c(M_k,M_\ell)>0$ for  $1\le k,\ell\le N$ with $k\ne\ell$. In this case \eqref{eq:prod} may be recast as  
\begin{equation}\label{eq:bound}
\|T^n-P_M\|\le Cr^n,\quad n\ge1,
\end{equation}
where $C=c(M_1,M_N)^{-1}$ and $r=\prod_{k=1}^Nc(M_k,M_{k+1})$, indices henceforth being considered modulo $N$. Since the asymptotic rate of convergence is determined by the value of $r\in(0,1]$, it is natural to seek the reordering of the subspaces $M_1,\dots,M_N$ which leads to the smallest possible value of $r$. More formally, given $N\ge2$ we let $S_N$ denote the symmetric group on $N$ letters and for each $\sigma\in S_N$ we let 
$r_\sigma=\prod_{k=1}^Nc(M_{\sigma(k)},M_{\sigma(k+1)}),$
so that for the reordered product $T_\sigma=P_{\sigma(N)}\cdots P_{\sigma(1)}$ we obtain
$$\|T^n_\sigma-P_M\|\le C_\sigma r_\sigma^n,\quad n\ge1,$$
where $C_\sigma=c(M_{\sigma(1)},M_{\sigma(N)})^{-1}$. The objective therefore is to find a permutation $\sigma\in S_N$ such that $r_\sigma=r_*$, where $r_*=\min\{r_\sigma:\sigma\in S_N\}$, and to find such a permutation a version of the following `greedy' algorithm was proposed in \cite[Section~9]{KaWe88}.

\begin{quote}\textsc{Greedy Algorithm:} 
Given $N\ge2$ independent closed subspaces $M_1,\dots,M_N$ of a Hilbert space $X$ whose mutual Friedrichs numbers are known we obtain permutations $\sigma_k\in S_N$, $1\le k\le N$, as follows. Let $\sigma_k(1)=k$ and for $j=2,\dots,N$ consider as possible values for $\sigma_k(j)$ any previously unused index $\ell$ which minimises $c(M_{\sigma_k(j-1)},M_\ell)$.  If at any stage there is more than one choice of such an index $\ell$ then proceed by considering all possible choices of this index and take $\sigma_k$ to be that permutation which among those leading to the least value of $r_{\sigma_k}$ comes first in the lexicographical ordering. Return the permutation $\sigma_G=\sigma_\ell$ where $\ell\in\{1,\dots,N\}$ is the smallest index such that $r_{\sigma_\ell}=\min\{r_{\sigma_k}:1\le k\le N\}$.
\end{quote}

If we let $r_G=r_{\sigma_G}$, $N\ge2$, then the Greedy Algorithm is correct if and only if $r_G=r_*$ for all constellations of subspaces.  By definition of $r_*$ it is clear that $r_*\le r_G$, $N\ge2$. In Section~\ref{sec:greed} we show that if the Greedy Algorithm were correct  then it would follow that $\mathrm{P}=\mathrm{NP}$. We then exhibit a simple example with $N=4$ in which $r_*<r_G$. Both  results are obtained as a consequence of a construction, presented in Section~\ref{sec:Fried}, which shows that any suitable collection of numbers in $[0,1]$  arises as the set of pairwise Friedrichs numbers between subspaces of some Hilbert space. This result is of independent interest and in particular implies that the problem of finding an optimal ordering is at least as hard as solving a multiplicative form of the Travelling Salesman Problem (TSP). In Section~\ref{sec:est} we  give sharp estimates for the maximal discrepancies between $r_*$ and $r_G$. In particular, we show that generically $r_G<\smash{r_*^{1/2}}$, and that the estimate is optimal in the sense that for every $\varepsilon\in(0,1)$ there exists some $N\ge2$ and a suitable collection of $N$ subspaces  of some Hilbert space such that $ r_G>(1-\varepsilon)\smash{r_*^{1/2}}.$ The last step once again requires the construction from Section~\ref{sec:Fried}.

\section{Friedrichs matrices}\label{sec:Fried}

Given $N\ge2$ closed subspaces $M_1,\dots,M_N$ of a  Hilbert space, we may consider the $N\times N$-matrix $(c(M_k,M_\ell))_{1\le k,\ell\le N}$ whose entries are 
the pairwise Friedrichs numbers between the various subspaces. We call the matrix arising in this way the \emph{Friedrichs matrix} corresponding to the collection of subspaces. It is clear that any Friedrichs matrix must be symmetric, have zeros along its main diagonal and elsewhere must have entries lying in the interval $[0,1]$. Is every square matrix which has these three properties a Friedrichs matrix for some collection of closed subspaces? The following result answers this question in the affirmative. Here and in what follows we use the same notation as in Section~\ref{sec:intro}.

\begin{thm}\label{thm:Fried}
Let $\FF\in\{\RR,\CC\}$ and $N\ge2$, and suppose that $C$ is an $N\times N$-matrix which is symmetric, has zeros along its main diagonal and elsewhere has entries lying in the interval $[0,1]$. Then there exists a  Hilbert space $X$ over the field $\FF$ and closed subspaces $M_1,\dots,M_N$ of $X$ such that $C$ is the corresponding  Friedrichs matrix. Furthermore, the subspaces can be constructed in such a way that $M_k\cap M_\ell=\{0\}$ for $1\le k,\ell\le N$ with $k\ne\ell$ and, if $N\ge3$, $P_kP_\ell P_m=0$ for $1\le k,\ell,m\le N$ mutually distinct.
\end{thm}

\begin{proof}
Let $C=(c_{k,\ell})$ and suppose first that $0\le c_{k,\ell}<1$ for $1\le k,\ell\le N$. Let $\{e_{k,\ell}:1\le k,\ell\le N,k\ne\ell\}$ be an orthonormal basis for the space $X=\FF^{N(N-1)}$ endowed with the Euclidean norm, and set
$$x_{k,\ell}=\begin{cases}
e_{k,\ell}, & 1\le k<\ell\le N,\\
c_{\ell,k}e_{\ell,k}+(1-c_{\ell,k}^2)^{1/2}e_{k,\ell}, & 1\le\ell< k\le N. 
\end{cases}$$
For $1\le k\le N$ let $B_k=\{x_{k,\ell}:1\le\ell\le N, \,\ell\ne k\}$, noting that these sets are orthonormal,  and consider the closed subspaces of $X$ given by $M_k=\Span B_k$. By our assumption that the entries of $C$ be strictly smaller than 1 we see that $M_k\cap M_\ell=\{0\}$ for $1\le k,\ell\le N$ with $k\ne\ell$, and in particular $M=\{0\}$. Furthermore, for $1\le k,k',\ell,\ell'\le N$ with $k\ne\ell$ and $k'\ne\ell'$ we have
$$(x_{k,\ell},x_{k'\ell'})=\begin{cases}
1,& k=k',\, \ell=\ell',\\
c_{k,\ell}, & k'=\ell, \,\ell'=k,\\
0, & \text{otherwise},
\end{cases}$$
from which it follows that $c(M_k,M_\ell)=c_{k,\ell}$ for $1\le k,\ell\le N$ with $k\ne\ell$ and, if $N\ge3$, that $P_kP_\ell P_m=0$ for $1\le k,\ell,m\le N$ mutually distinct. 

Now consider the general case where $0\le c_{k,\ell}\le1$ for $1\le k,\ell\le N$ and consider the matrix $B=(b_{k,\ell})$ with entries
$$b_{k,\ell}=\begin{cases}
c_{k,\ell}, & c_{k,\ell}<1,\\
0, & c_{k,\ell}=1,
\end{cases}$$
for $1\le k,\ell\le N$. By the first part we may find closed subspaces $L_1,\dotsc,L_N$ of $\FF^{N(N-1)}$ whose Friedrichs matrix is $B$. Let  $X=\FF^{N(N-1)}\oplus Y$, where $Y=\bigoplus_{1\le \ell<m\le N} \ell^2$, and endow $X$ with its natural Hilbert space norm. Moreover, let $U,V$ be two closed subspaces of $\ell^2$ such that $U+V$ is not closed. For $1\le k,\ell,m\le N$ with $\ell< m$  define the subspaces $Y^{\ell,m}_k$ of $\ell^2$ by
$$Y^{\ell,m}_k=\begin{cases}
U, & c_{\ell,m}=1 \text{ and } k=\ell,\\
V, & c_{\ell,m}=1 \text{ and } k=m,\\
\{0\}, & \text{otherwise},
\end{cases}$$
and for $1\le k\le N$  define the closed subspace $M_k$ of $X$ by  $M_k=L_k\oplus Y_k$, where $Y_k= \bigoplus_{1\le\ell<m\le N}Y^{\ell,m}_k$. If $1\le k<\ell\le N$ are such that $c_{k,\ell}<1$, then for $1\le m<n\le N$ we have either $Y_k^{m,n}=\{0\}$ or $Y_\ell^{m,n}=\{0\}$ and therefore $c(M_k,M_\ell)=c(L_k,L_\ell)=b_{k,\ell}=c_{k,\ell}$. Suppose that $1\le k<\ell\le N$ and  that $c_{k,\ell}=1$. Then for $1\le m<n\le N$ we see that $Y^{m,n}_k+Y_\ell^{m,n}=U+V$ if and only if $k=m$ and $\ell=n$, and that otherwise $Y^{m,n}_k+Y_\ell^{m,n}$ equals either $U,V$ or $\{0\}$. It follows that $Y_k+Y_\ell$ is not closed, and hence $M_k+ M_\ell$ is not closed. By \cite[Theorem~9.35]{De01} this implies that $c(M_k,M_\ell)=1=c_{k,\ell}$, and hence we have the required subspaces. Moreover, it is clear from the construction that $M_k\cap M_\ell=\{0\}$ for $1\le k,\ell\le N$ with $k\ne\ell$ and, if $N\ge3$, that $P_kP_\ell P_m=0$ for $1\le k,\ell,m\le N$ mutually distinct.
\end{proof}

\begin{rem}
Note that the result in particular provides a new proof of the fact that in general the optimal value of $r$ in \eqref{eq:bound} cannot be expressed as a function of pairwise Friedrichs numbers between the subspaces $M_1,\dots,M_N$ when $N\ge3$, as was first observed in a particular case in \cite[Example~3.7]{DeHu97}. Indeed, for any collection of closed subspaces $M_1',\dots,M_N'$, $N\ge3$, of some Hilbert space such that in the method of alternating projections we do not have  convergence in one step, by Theorem~\ref{thm:Fried} we may find an alternative collection of closed subspaces $M_1,\dots,M_N$ of some Hilbert space with the same  pairwise Friedrichs numbers but for which $T=P_M=0$.
\end{rem}

\section{Incorrectness of the Greedy Algorithm}\label{sec:greed}

In this section we turn to the Greedy Algorithm presented in Section~\ref{sec:intro}, and in particular we ask whether the algorithm is \emph{correct} in the sense that the ordering it produces leads to the optimal value of $r\in[0,1]$ in \eqref{eq:bound}. We first consider the connection between our problem of finding an optimal ordering and the classical TSP, and we show in Corollary~\ref{cor:TSP} below that correctness of the Greedy Algorithm for a sufficiently large class of cases would imply that $\mathrm{P}=\mathrm{NP}$. We then exhibit a simple example in which the Greedy Algorithm gives a suboptimal ordering.

Recall that in the graph-theoretical formulation of the TSP we are given, for some $N\ge2$,  a complete graph $K_N$ with vertices $V_N=\{1,2,\dotsc,N\}$ and  a \emph{weight function} 
$$w\colon \big\{(k,\ell)\in V_N^2:k\ne\ell\big\}\to\RR$$
 such that  $w(k,\ell)=w(\ell,k)$ for $1\le k,\ell\le N$ with $k\ne\ell$, and the objective is to find a permutation $\sigma^*\in S_N$ such that $\Sigma_{\sigma^*}=\min\{\Sigma_\sigma:\sigma\in S_N\}$, where for a permutation $\sigma\in S_N$ we let 
$$\Sigma_\sigma=\sum_{k=1}^Nw\big(\sigma(k),\sigma(k+1)\big)$$
with indices, as usual,  considered modulo $N$. We will be interested primarily in the multiplicative form of the TSP, denoted by MTSP, in which the objective is to minimise not the additive cost but instead to find $\sigma^*\in S_N$ such that $\Pi_{\sigma^*}=\min\{\Pi_\sigma:\sigma\in S_N\}$, where for a permutation $\sigma\in S_N$ we let
$$\Pi_\sigma=\prod_{k=1}^Nw\big(\sigma(k),\sigma(k+1)\big).$$
It is clear that TSP and MTSP have the same solution, and indeed one may pass from one form of the problem to the other simply by replacing the weight function by its logarithm or its exponential, as appropriate. Furthermore, the solution of TSP is unaffected by shifting the values of the weight function by a constant amount, which implies in particular that there is no loss of generality in considering the MTSP only for weight functions taking values in the range $[0,1]$.

It is well known that the TSP, and hence also MTSP, is NP-complete. This means that it lies in the complexity class NP and is NP-hard, which is to say that any other problem in NP can be transformed into an instance of the TSP in polynomial time. Furthermore, by considering the corresponding decision problems it can be seen that TSP and hence MTSP remain NP-complete if the weight function is assumed to take distinct values on distinct pairs. Our first result is an application of Theorem~\ref{thm:Fried} showing that the subspace ordering problem is NP-hard.

\begin{prp}\label{thm:TSP}
The problem of finding an optimal ordering for    collections of independent closed subspaces with pairwise distinct Friedrichs numbers is NP-hard.
\end{prp}

\begin{proof}
It suffices to show that every instance of TSP with distinct costs can be transformed in polynomial time into a subspace ordering problem with pairwise distinct Friedrichs numbers. However, this follows straightforwardly from Theorem~\ref{thm:Fried}. Indeed, given a TSP problem on $N\ge2$ vertices we may transform it to an instance of MTSP with weight function taking values in the range $[0,1]$ in $O(N^2)$ steps. Let $C=(c_{k,\ell})_{1\le k,\ell\le N}$ be the symmetric matrix with zeros along its main diagonal  and entries $c_{k,\ell}=w(k,\ell)$ for $1\le k,\ell\le N$ with $k\ne\ell$. By Theorem~\ref{thm:Fried} there exists a Hilbert space $X$ and independent closed subspaces $M_1,\dots,M_N$ of $X$ such that $C$ is the associated Friedrichs matrix. Moreover, it is clear from the proof of Theorem~\ref{thm:Fried}  that it is possible to obtain these subspaces in polynomial time. If we find a permutation $\sigma^*\in S_N$ such that $r_{\sigma^*}=r_*$, then since $r_\sigma=\Pi_\sigma$ for all $\sigma\in S_N$ the permutation $\sigma^*$ also solves our instance of MTSP, and hence the original TSP problem. Since TSP is known to be NP-hard,  our problem is too.
\end{proof}

\begin{rem}
Note that the subspaces $M_1,\dots,M_N$ are not merely independent but satisfy the much stronger conditions described in Theorem~\ref{thm:Fried}. In particular, the result remains true if the subspaces which we are trying to order are merely pairwise quasi-disjoint in the sense of Section~\ref{sec:intro}.
\end{rem}

The result shows that the existence of any polynomial-time algorithm which solves the subspace ordering problem in a sufficiently large number of cases implies that $\mathrm{P}=\mathrm{NP}$.   In particular, we obtain the following consequence for the Greedy Algorithm.

\begin{cor}\label{cor:TSP}
Correctness of the Greedy Algorithm for independent subspaces with pairwise distinct Friedrichs numbers implies that $\mathrm{P}=\mathrm{NP}$.
\end{cor}

\begin{proof}
It is straightforward to see that if all the pairwise Friedrichs numbers are distinct then the Greedy Algorithm terminates after $O(N^3)$ steps, where $N\ge2$ is the number of subspaces we a required to order optimally.
\end{proof}

\begin{rem}
The version of the Greedy Algorithm formulated in \cite[Section~9]{KaWe88} differs from ours in that it does not consider all possible greedy paths and hence runs in polynomial time even if the pairwise Friedrichs numbers are not assumed to be distinct. Note also that, as in the case of Proposition~\ref{thm:TSP}, the assumption of  independence on the subspaces can be relaxed to pairwise quasi-disjointness.
\end{rem}

Given that the question whether $\mathrm{P}=\mathrm{NP}$  is a long-standing open problem, one may view Proposition~\ref{thm:TSP} as evidence suggesting that the Greedy Algorithm does not in general lead to an optimal ordering of the subspaces in question. This is indeed the case, as the following example illustrates.

\begin{ex}\label{ex}
Let $\FF\in\{\RR,\CC\}$ and let $X=\FF^4$ with the Euclidean norm. Consider the one-dimensional subspaces   $M_k=\Span\{x_k\}$, $1\le k\le 4$, where $x_1,\dotsc,x_4\in X$ are the unit vectors 
$$\begin{aligned}
x_1&=(1,0,0,0),\\
x_2&=\bigg(\frac12,\frac{\sqrt{3}}{2},0,0\bigg),\\
x_3&=\bigg(\frac15,\frac{1}{10\sqrt{2}}, \frac{\sqrt{191}}{10\sqrt{2}},0\bigg),\\
x_4&=\bigg(\frac{1}{10},\frac{1}{8}, \frac{1}{15},\frac{\sqrt{13967}}{120}\bigg).
\end{aligned}$$
The Friedrichs numbers satisfy $c(M_k,M_\ell)=|(x_k,x_\ell)|$ for $1\le k,\ell\le 4$ with $k\ne\ell$, so the associated Friedrichs matrix is given (approximately) by
$$C=
\left(
\begin{array}{cccc}
  0.0000& 0.5000  & 0.2000 & 0.1000 \\
  0.5000& 0.0000  & 0.1612  &0.1583\\
0.2000  & 0.1612  & 0.0000  & 0.0940\\
 0.1000 &  0.1583 &0.0940   &0.0000
\end{array}
\right).$$
The permutation $\sigma_G\in S_4$ produced by the Greedy Algorithm is 
$$\big(\sigma_G(k)\big)_{k=1}^4=(1,4,3,2),$$
which leads to $r_G\approx7.5772\times 10^{-4}$. The permutation $\sigma\in S_4$ given by 
$$\big(\sigma(k)\big)_{k=1}^4=(1,4,2,3)$$
leads to the optimal value $r_*=r_\sigma\approx 5.1033\times10^{-4}$, and in particular $r_G>r_*$. It follows that the Greedy Algorithm is not correct.
\end{ex}

\begin{rem}
Example~\ref{ex} disproves a claim made in \cite[Section~9]{KaWe88}, namely that the Greedy Algorithm always leads to an optimal ordering in the case of independent subspaces. The examples considered in  \cite[Section~9]{KaWe88} involve only $N=3$ subspaces, a special case in which the Greedy Algorithm performs an exhaustive search of all possible orderings (up to the direction in which they are traversed) and in particular is correct. Thus Example~\ref{ex} is minimal in terms of the number of subspaces involved.
\end{rem}

\section{Sharp estimates for the degree of suboptimality}\label{sec:est}

Having shown in Section~\ref{sec:greed} that the Greedy Algorithm does not in general lead to an optimal ordering of the subspaces in the method of alternating projections, we seek now to quantify how much the result reached by the Greedy Algorithm can disagree with the optimal result.  Given a collection of closed subspaces of a Hilbert space such that at least one of the pairwise Friedrichs numbers is zero, we see that for suitable orderings of the subspaces we obtain convergence after at most two steps in the method of alternating projections.  Another essentially uninteresting case for asymptotic analysis is when all of the pairwise Friedrichs numbers equal 1, so that no ordering leads to a useful estimate in \eqref{eq:prod}.  If either of these two cases holds we shall say that the collection of subspaces involved is \emph{non-generic}, and otherwise we call it \emph{generic}.

\begin{thm}\label{thm:error}
Let $N\ge2$ and suppose that $M_1,\dotsc,M_N$ are closed subspaces of a Hilbert space $X$. Then 
\begin{equation}\label{eq:bounds}
r_*\le r_G\le r_*^{1/2}.
\end{equation}
Moreover, the second inequality is strict unless the collection $M_1,\dotsc,M_N$ of subspaces is non-generic
\end{thm}

\begin{proof}
For $1\le k\le N$ let $\sigma_k\in S_N$ be the permutation produced by running the Greedy Algorithm with the starting vertex $\sigma_k(1)=k$ and let $r_k=r_{\sigma_k}$. Then certainly $r_*\le r_k$ for $1\le k\le N$, and hence also $r_*\le r_G$.  For $1\le k,\ell\le N$ let 
$$s_k(\ell)=\sigma_k\big(\sigma_k^{-1}(\ell)+1\big)$$
denote the index of the successor to $M_\ell$ in the ordering of the subspaces determined by $\sigma_k$, noting that $s_k(\ell)=1$ if $\sigma_k(\ell)=N$. Let $\sigma\in S_N$ and for $1\le k,\ell\le N$ with $k\ne\ell$ let $w(k,\ell)=c(M_k,M_\ell)$. Let $1\le k,\ell\le N$. If $\sigma_k^{-1}(\sigma(\ell))< \sigma_k^{-1}(\sigma(\ell+1))$, which is to say that in the ordering determined by $\sigma_k$ the subspace $M_{\sigma(\ell)}$ comes before $M_{\sigma(\ell+1)}$, then by definition of the Greedy Algorithm we must have
$$w\big(\sigma(\ell),s_k(\sigma(\ell)\big)\le w\big(\sigma(\ell),\sigma(\ell+1)\big),$$
while if  $\sigma_k^{-1}(\sigma(\ell))> \sigma_k^{-1}(\sigma(\ell+1))$ then 
$$w\big(\sigma(\ell+1),s_k(\sigma(\ell+1)\big)\le w\big(\sigma(\ell),\sigma(\ell+1)\big).$$ 
Since $w$ takes values in $[0,1]$ it follows that
\begin{equation}\label{eq:prod2}
w\big(\sigma(\ell),s_k(\sigma(\ell)\big) w\big(\sigma(\ell+1),s_k(\sigma(\ell+1)\big)\le w\big(\sigma(\ell),\sigma(\ell+1)\big)
\end{equation}
for $1\le k,\ell\le N$. Thus for $1\le k\le N$ we have
\begin{equation}\label{eq:calc}
\begin{aligned}
r_k^2&=\prod_{\ell=1}^N c(M_{\sigma_k(\ell)},M_{\sigma_k(\ell+1)})^2
= \prod_{\ell=1}^Nw\big(\sigma_k(\ell),\sigma_k(\ell+1)\big)^2\\
&=\prod_{\ell=1}^N\Big(w\big(\sigma(\ell),s_k(\sigma(\ell)\big) w\big(\sigma(\ell+1),s_k(\sigma(\ell+1)\big)\Big)\\
&\le \prod_{\ell=1}^Nw\big(\sigma(\ell),\sigma(\ell+1)\big)=\prod_{\ell=1}^N c(M_{\sigma(\ell)},M_{\sigma(\ell+1)})=r_\sigma.
\end{aligned}
\end{equation}
Since $\sigma\in S_N$ was arbitrary we deduce that $r_k^2\le r_*$ for $1\le k\le N$, and in particular $r_G^2\le r_*$, as required.

Now suppose that $r_G^2= r_*$, and let $\sigma_*\in S_N$ be a permutation such that $r_{\sigma_*}=r_*$. Since $r_G^2\le r_k^2 \le r_*$ for $1\le k\le N$, we see that in fact $r_k^2=r_*$ for $1\le k\le N$. Now either one of the pairwise Friedrichs numbers is zero or all of the pairwise Friedrichs numbers are non-zero. In the latter case it is clear from \eqref{eq:calc} that we must have equality in \eqref{eq:prod2} for $1\le k,\ell\le N$ when $\sigma=\sigma_*$. Taking $k=\sigma_*(\ell)$ in \eqref{eq:prod2} for $1\le\ell\le N$, it follows that 
$$w\big(\sigma_*(\ell),\sigma_*(\ell+1)\big)=\min\big\{w(\sigma_*(\ell),k):1\le k\le N, \,k\ne\sigma_*(\ell)\big\}$$
for $1\le\ell\le N$. It follows that $\sigma_*$ is itself a permutation considered by the Greedy Algorithm, and therefore $r_*= r_G$. Hence $r_*^2=r_*$, and since $r_*\ne0$ we have $r_*=1$,  which implies that $c(M_k,M_\ell)=1$ for $1\le k,\ell\le N$ with  $k\ne\ell$. It follows that $r_G^2<r_*$ unless the collection $M_1,\dotsc,M_N$ of subspaces  is  non-generic.
\end{proof}

It remains to be investigated to what extent the second bound in \eqref{eq:bounds} is sharp for generic constellations of subspaces. Our final example shows that it cannot be improved in the sense that given any $\varepsilon\in(0,1)$ there exists a generic constellation of subspaces of some Hilbert space such that 
$$r_G>(1-\varepsilon)r_*^{1/2}.$$ 
In fact, there exists a constellation of $N$ such subspaces for every even $N\ge4$.

\begin{ex}
Given a positive a positive integer $n\ge2$, let $N=2n$ and suppose that $0<\delta<c<1 $. By Theorem~\ref{thm:Fried} there exists a Hilbert space $X$ and a generic constellation $M_1,\dotsc,M_N$ of closed subspaces of $X$ such that for $1\le k,\ell\le N$ with $k\ne\ell$ we have
$$c(M_k,M_\ell)=\begin{cases}
c &\text{if $k=\ell\pm1\!\!\!\!\pmod{N}$}\\
c\delta &\text{if $k=\ell\pm2\!\!\!\!\pmod{N}$ and $k$ is even,}\\
1 & \text{otherwise}.
\end{cases}$$
Let $\sigma_0\in S_N$ denote the identity permutation. Then $r_{\sigma_0}=c^N$. If we think of the subspaces as the vertices of a complete graph of order $N$, and we let the edges have weights given by the pairwise Friedrichs numbers, then $r_\sigma\ge r_{\sigma_0}$ for all permutations $\sigma\in S_N$  involving no $c\delta$-edges. Moreover, any cycle $\sigma\in S_N$ which uses at least one of the $c\delta$-edges cannot use more than $n-1$ of them, and must involve at least two 1-edges, so for any such cycle 
$$r_\sigma\ge c^{n-1}(c\delta)^{n-1}=c^{N-2}\delta^{n-1}.$$ In particular, if $c^2\le\delta^{n-1}$ then $r_*=r_{\sigma_0}$. It is easy to that 
$$r_G\ge c^2(c\delta)^{n-1}=c\delta^{n-1}r_*^{1/2}.$$
Given $\varepsilon\in(0,1)$ we deduce that $r_G>(1-\varepsilon)r_*^{1/2}$ provided $c,\delta\in(0,1)$ are such that $c^2\le \delta^{n-1}$ and $c\delta^{n-1}>1-\varepsilon$. These conditions are satisfied for instance when  $(1-\varepsilon)^{1/3}<c<1$ and $\delta =c^{2/(n-1)}$. Furthermore, it is the case that for any $r,\varepsilon\in(0,1)$ there exist generic constellations of $N$ subspaces of a Hilbert space for all sufficiently large even $N\ge4$ with the properties that   $r_G>\smash{(1-\varepsilon)r_*^{1/2}}$ and $r_*=r$.
\end{ex}

\section{Acknowledgements}
For financial support O.D.\ thanks Magdalen College, Oxford, A.J.\ thanks the Mathematical Institute of the University of Oxford, S.R.\ and R.S.\ thank both St John's College, Oxford, and the Mathematical Institute, and L.S.\ thanks the EPSRC. All authors would further like to express their thanks to Alexis Chevalier, Stefan Kiefer, Dominik Peters and Zhixuan Wang for useful discussions. 

\bibliographystyle{plain}

\end{document}